\documentclass{amsart}
\usepackage{graphics}
\usepackage{latexsym}
\usepackage{amsfonts}
\usepackage{color}
\usepackage{enumerate}
\usepackage{graphicx}

\newtheorem{thm}{Theorem}[section]

\newtheorem{prop}[thm]{Proposition}
\newtheorem{lem}[thm]{Lemma}
\newtheorem{rem}[thm]{Remark}
\newtheorem{defn}[thm]{Definition}

\newcommand\specialref{}
\newtheorem*{thmx}{Theorem \specialref}

\newenvironment{thmn}[1]
  {\renewcommand\specialref{#1}\thmx}
  {\endthmx}

\def\x{\underline x}
\def\y{\underline y}

\def\diam{\text{diam}}
\def\supp{\text{supp}}

\def\N{\mathbb {N}}

\def\B{\mathcal {B}}
\def\eps{\varepsilon}
\def\ie{{\em i.e.,\ }}
\def\eg{{\em e.g.\ }}
\def\TF{{\mathcal L}}
\def\FF{{\mathcal P}}
 \newcommand{\set}[1]{\left\{#1\right\}}

\numberwithin{equation}{section}

\author[H. Bruin]{Henk Bruin}
\author[P. Oprocha]{Piotr Oprocha}
\address[H. Bruin]{Faculty of Mathematics,
University of Vienna,
Oskar Morgensternplatz 1, A-1090 Vienna,
Austria} \email{henk.bruin@univie.ac.at}
\address[P. Oprocha]{Faculty of Applied
Mathematics, AGH University of Science and Technology, al.
Mickiewicza 30, 30-059 Krak\'ow, Poland} \email{oprocha@agh.edu.pl}

\title{On "observable" Li-Yorke tuples for interval maps}
\thanks{December 2 2013 -- compiled \today}
\begin{document}

\subjclass[2010]{Primary 37E05; Secondary 37A25, 37A40, 37B05.}

\keywords{Li-Yorke pair, proximal, limsup full, non-singular, Manneville-Pomeau map, Perron-Frobenius operator, first return map}
\maketitle

\begin{abstract}
In this paper we study the set of Li-Yorke $d$-tuples and its $d$-dimensional Lebesgue measure for interval maps
$T\colon [0,1] \to [0,1]$. If a topologically mixing $T$ preserves an absolutely continuous probability measure 9with respect to Lebesgue), then the $d$-tuples
have Lebesgue full measure, but if $T$ preserves an infinite absolutely continuous measure,
the situation becomes more interesting.
Taking the family of Manneville-Pomeau maps as example, we show that for any $d \ge 2$, it is possible
that the set of Li-Yorke $d$-tuples has full Lebesgue measure, but the set of Li-Yorke $d+1$-tuples
has zero Lebesgue measure.
\end{abstract}

\section{Introduction}

Abundance of Li-Yorke pairs (see \cite{LY} and Definition~\ref{def:liyorke} below) is a frequently used
criterion for declaring that a dynamical
system $T\colon X \to X$ on a compact metric space $(X,\rho)$ is chaotic.
A milestone was the result that positive topological entropy
implies the existence of an uncountable scrambled set, \ie a set in
which every pair of distinct points has the Li-Yorke property \cite{BGKM}.
However, $2^\infty$ maps (\ie with periodic points of period $2^k$ for each $k \ge 0$, but no other periods)
have zero entropy, but they can still have uncountable scrambled sets \cite{Smital},
so that Li-Yorke pairs give a slightly more refined view on mathematical chaos than
the condition that $h_{top}(T) > 0$.

Multimodal interval maps never have closed invariant scrambled sets,
and it can be proved that $C^3$ multimodal interval maps with
nonflat critical points only have scrambled sets of Lebesgue measure zero,
see  \cite{BJ}.
 In view of these results, when speaking about "observable chaos" it is more reasonable to consider the size of the
set of Li-Yorke pairs, than scrambled sets themselves.
This idea comes from Lasota, and was first employed by Pi\'orek in \cite{Piorek}.

Various papers (see \eg \cite{BHR, BJ}) comment on the measure-theoretic properties
of Li-Yorke pairs and scrambled sets. In particular, \cite{BJ}
gives a comprehensive account in the setting of smooth  multimodal
interval maps $T\colon [0,1] \to [0,1]$, on questions
whether Li-Yorke pairs have full two-dimensional Lebesgue measure
$\lambda_2 := \lambda \times \lambda$ in $[0,1]^2$.
Under certain mixing conditions of $\lambda$ (exactness suffices
but we use the weaker condition of $\limsup$ full, see Definition~\ref{DefLimsupFull})
this is indeed the case.
One question left open in \cite{BJ} is whether
there are smooth conservative maps for which
$\lambda_2$-a.e. pair $(x,y)$ is Li-Yorke and additionally its orbit under $T\times T$ is not dense in $[0,1]^2$.

The topic of this paper is the $d$-fold measure of Li-Yorke $d$-tuples.
We give the definitions first.
Let $(X,\rho)$ be a compact metric space and $T\colon X \to X$ a continuous map acting on it.
\begin{defn}\label{def:liyorke}
A $d$-tuple $\x = (x_1, \dots, x_d)$ is called:
\begin{enumerate}
\item {\em asymptotic} if
$\lim_n \max_{i, j} \rho(T^n(x_i), T^n (x_j)) = 0$;
\item {\em proximal} if
$\liminf_n \max_{i, j} \rho(T^n(x_i), T^n (x_j)) = 0$;
\item {\em $\delta$-separated} if
$\limsup_n \min_{i \neq j} \rho(T^n(x_i), T^n (x_j)) > \delta$;
if $\x$ is $\delta$-separated for some $\delta>0$, we call it \emph{separated};
\item {\em Li-Yorke (LY for short)} if it proximal and separated, that is:
$$
\begin{cases}
\liminf_n \max_{i, j} \rho(T^n(x_i), T^n (x_j)) = 0, \\[1mm]
\limsup_n \min_{i\neq j} \rho(T^n(x_i), T^n (x_j)) > 0.
\end{cases}
$$
\item {\em $\delta$-Li-Yorke (or simply, $\delta$-LY)} for some $\delta>0$, if $\x$ is $LY$ and:
$$
\limsup_n \min_{i\neq j} \rho(T^n(x_i), T^n (x_j)) > \delta.
$$
\end{enumerate}
We denote by $LY_d$ and $LY_d^\delta$ the set of all $LY$ and $\delta$-$LY$
$d$-tuples, respectively. Also we use $T_d$ as abbreviation
for the $d$-fold product map $T \times \dots \times T$ on $X^d$.
\end{defn}

It is known that a transitive system with a fixed point has a LY $d$-tuple for any $d\ge 2$ \cite{Xiong},
and consequently, every totally transitive system with dense periodic points (hence topologically weakly mixing)
must have such tuples.
It is also not hard to see that $T$ has LY $d$-tuples if and only if $T^n$ has them for every $n\ge 1$. Therefore,
each transitive map on the interval must have LY $d$-tuples. In fact, every
topologically mixing map on an infinite space
has a dense Mycielski set $M$  such that any $d\ge 2$ distinct points in
$M$ form a LY $d$-tuple (see \eg \cite{LO}).
On the other hand, maps of the interval with zero topological entropy never have LY $3$-tuples \cite{Li}.
However, there are dynamical systems on the
Cantor set such that each $d$-tuple
of distinct points is LY, but no uncountable set with this property for $(d+1)$-tuples exists \cite{LO}.
In fact, the system need not have LY $(d+1)$-tuples at all \cite{Do}.

The main motivation of the paper is the following question.\\
\begin{quote}
{\bf Question:} \emph{How large is the set of LY $d$-tuples?}\\
\end{quote}
Obviously, there is no one good answer to this question without specifying what "large" means.
In purely topological case it would be a residual set. But in many cases,
\eg on the interval, there is a natural reference measure such as Lebesgue
measure.
Even with this natural tool at hand, the answers depend on the degree of smoothness of the map.
The smoother the map is, the better the proposed method of measurement is appropriate.

Let $\lambda$ be Lebesgue measure, or more generally a non-singular Borel reference measure.
We will assume that $\lambda$ is fully supported, \ie $\lambda(U) > 0$ for every open $U \subset X$, or otherwise assume that $\lambda$ is
non-atomic and restrict $T$ to $\supp(\lambda)$.
Let $\lambda_d = \lambda \times \dots \times \lambda$ denote the $d$-fold
product measure on $X^d$.

The following definitions come from \cite{Bar} and \cite{PB82}
respectively:

\begin{defn}\label{DefLimsupFull}
Let $\lambda$ be a non-singular probability measure on $X$.
Then $\lambda$ is called:
\begin{enumerate}
\item {\em $\limsup$ full} if $\lambda(A) > 0$ implies
that $\limsup_{n\to\infty} \lambda(T^n(A)) = 1$;
\item {\em full}\footnote{Compared to $\limsup$ full, $\lim$ full seems a more
logical name for this property, but it is called
{\em full}
%by Rokhlin \cite{Ro64} and
by Proppe \& Boyarski in \cite{PB82}, hence we follow this terminology.}
if $\lambda(A) > 0$ implies
that $\lim_{n\to\infty} \lambda(T^n(A)) = 1$.
\end{enumerate}
\end{defn}

When $d \ge 3$, then $\lambda$ being $\limsup$ full is no longer
sufficient to guarantee that $\lambda_d$-a.e.\ $d$-tuple is Li-Yorke.
Instead, if $\lambda$ admits an equivalent weak mixing $T$-invariant probability measure $\mu$, then this holds, see \cite{BHR} and
Lemma~\ref{lem:weak_mixing}.
We show
\begin{thmn}{A}\label{t2}
%\begin{thmref}{t2}
Let $\lambda$ be a non-singular, fully supported, Borel probability
measure, and denote by $\lambda_d$ the $d$-fold product measure.
\begin{enumerate}
\item\label{t2:c1} If $\lambda$ is $\limsup$ full then $\lambda_2(LY^\delta_2)=1$ for any $\delta<\diam(X)/2$,
\item\label{t2:c2} If $\lambda$ is full then $\lambda_d(LY_d^\delta)=1$ for
every $d \ge 2$ and some $\delta>0$. If $X$ is additionally connected then
$\lambda_d(LY_d^\delta)=1$ for every $\delta<\diam(X)/2(d-1)$.
\end{enumerate}
%\end{thmref}
\end{thmn}

\begin{rem} Without the connectedness assumptions, a bound
$\delta<\diam(X)/2(d-1)$ cannot work. For instance, if $X$ is a
union of two
small intervals but placed at long distance, then $\diam (X)$ is large
but two points in any triple have to be very close.
\end{rem}

For a smooth topologicallly mixing interval map $T$ preserving a
probability measure $\mu \ll \lambda$,
the above theorem supplies an abundance of Li-Yorke $d$-tuples.
If the $T$-invariant measure $\mu$ is only $\sigma$-finite,
then the difficulty in showing the abundance of Li-Yorke
$d$-tuples for $d \ge 3$ lies in the separation along a subsequence.
Under mild conditions, any two points in a $d$-tuple separate
infinitely often, but it is difficult to show that three or more points
separate at the same time.
The family of \emph{Manneville-Pomeau maps} $T_\alpha:[0,1] \to [0,1]$
defined by
$$
T_\alpha(x) = \begin{cases}
x(1+2^\alpha x^\alpha) & \text{ if } x \in [0,\frac12), \\
2x-1 & \text{ if } x \in [\frac12,1]. \\
\end{cases}
$$

\begin{figure}[htb]
\begin{center}
\includegraphics[width=0.3\textwidth]{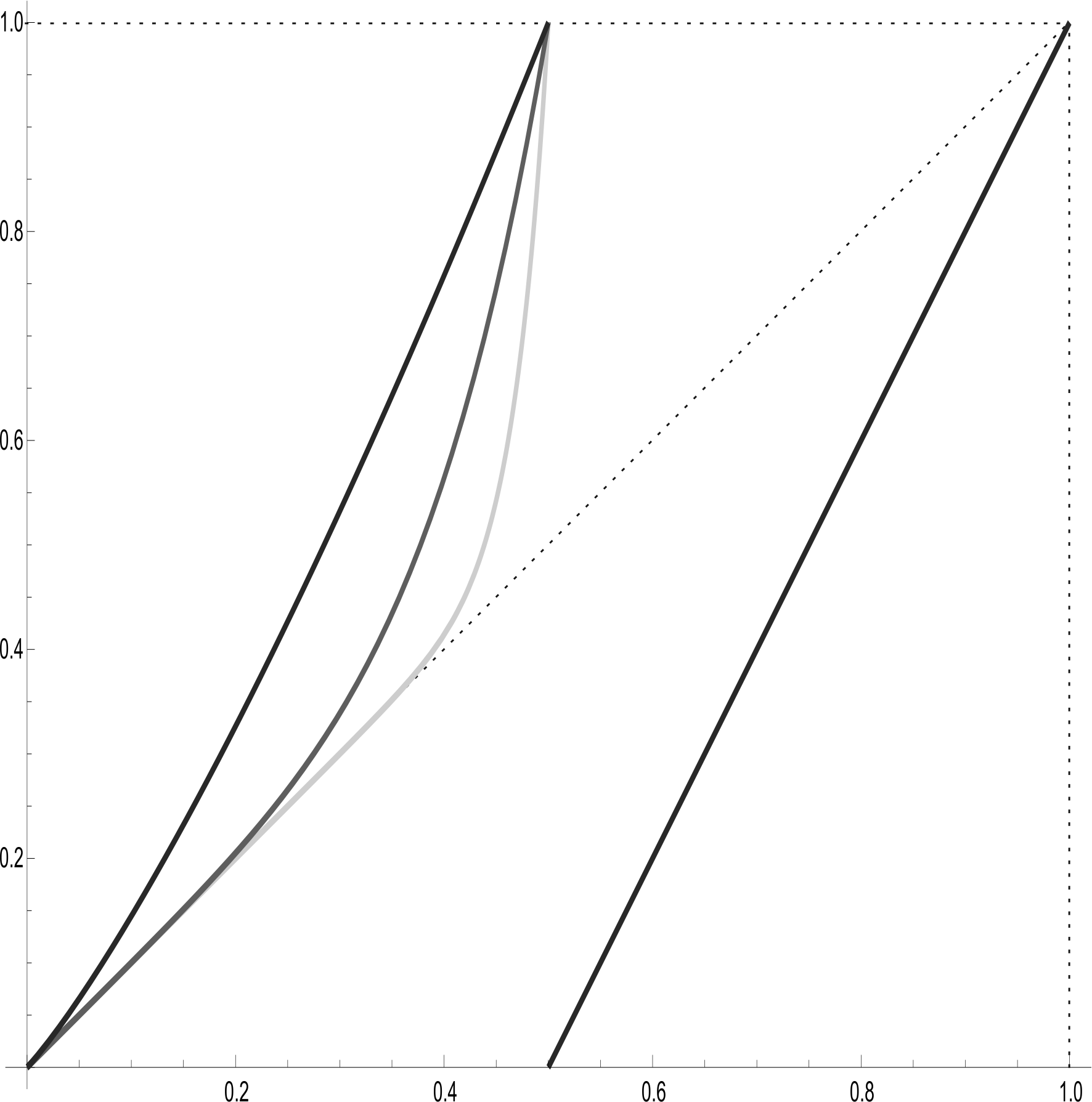}
\caption{Graph of $T_{\frac{1}{2}}$, $T_4$ and $T_{15}$.}
\label{fig:linear_chain}
\end{center}
\end{figure}

\noindent
illustrates this perfectly.

% {\color{red}P: in one of the arguments we use the fact that points at different sides on 1/2 are separated in next iteration.
%It is no longer case when slope is reversed. So maybe it is worth to say something about this map also in section 5? Results on $\mu$ seem
%to be valid there.}
The point $0$ is a neutral fixed point, weakly expanding, but the speed
at which points move away from $0$ is slower as $\alpha$ increases.
The typical situation  can thus be that one point in a $d$-tuple
separates itself from the rest, while the other $d-1$ points
linger in a neighbourhood of $0$.
Using known techniques of renewal theory
to a first return induced map we prove the following theorem.

\begin{thmn}{B}
For the Manneville-Pomeau map $T_\alpha$, the following conditions hold:
\begin{enumerate}
\item $\lambda_2$-a.e.\ pair is $\frac13$-Li-Yorke,
\item and for $d \ge 3$:
\begin{enumerate}[(i)]
\item if $\alpha \le \frac{d-1}{d-2}$ and $\delta < 1/2(d-2)$, then $\lambda_d$-a.e.\ $d$-tuple is $\delta$-Li-Yorke;
\item if $\alpha > \frac{d-1}{d-2}$ and $\eps > 0$, then
$\lambda_d$-a.e.\ $d$-tuple is not $\eps$-separated;
in particular $\lambda_d$-a.e.\ $d$-tuple is not Li-Yorke.
\end{enumerate}
\end{enumerate}
\end{thmn}

\begin{rem} For $\alpha > 2$, the product system is $\lambda_2$-dissipative,
whence orbits of typical pairs are not dense in $[0,1]^2$, however still $\lambda_2$-a.e.\ pair is LY. This addresses a question posed in \cite[p.~527]{BJ}.

If $\frac{d}{d-1} < \alpha \le \frac{d-1}{d-2}$,
then $\lambda_d$-a.e.\ $\x$ is LY in $[0,1]^2$, however $\lambda_{d+1}$-a.e\ $\x$ is not LY,
nor asymptotic (\ie $T^n_{\alpha,d}(\x) \not\to \Delta$, the diagonal).
\end{rem}

When measuring tuples using Lebesgue measure as reference, much may depend on the class of maps we are considering.
In the class of continuous maps we can quite easily perturb the dynamics, and hence topological conjugacy can completely change
qualitative description of the map in terms of reference measure.
We can prove the following:

\begin{thmn}{C}
There exist pairwise topologically conjugate maps $P, S, T\colon [0,1] \to [0,1]$
and sets $K, L,M$ such that:
\begin{enumerate}
\item $K,L,M$ have full Lebesgue measure;
\item there are positive numbers $\delta_d$ such that every $d$-tuple of distinct points in $L$ (resp. in $M$)
is $\delta_d$-LY for $S$ (resp.\ $T$);
\item $\omega(x,S) = [0,1]$ for every $x \in L$;
\item\label{anota:c4} there exists a Cantor minimal set $A$ of positive measure such that $\omega(x,T)=A$ for every $x\in M$.
\item none of the pairs $(x,y)\in K\times K$ is LY for $P$; in particular the set of LY pairs  has zero Lebesgue measure.
\end{enumerate}
\end{thmn}

The maps $P,S,T$ are the same from dynamical point of view, however the size of the set of LY tuples detected by Lebesgue measure
is completely different in each case.

The paper is organised as follows. Section~\ref{sec:PrelimMeas} gives preliminary information from ergodic theory.
Section~\ref{sec:dfold}  shows the results on LY $d$-tuples that can be derived from the assumption
that $\lambda$ is full or $\limsup$ full.
In Section~\ref{sec:continuous}  we present and prove our results for the $C^0$ setting.
Finally, in Section~\ref{sec:MP}  we discuss the situation of $\sigma$-finite measure (Manneville-Pomeau)
which provides the most interesting examples where the theory of Section~\ref{sec:dfold} break down.

%\section{Basic notation}

\section{Preliminaries from Measure Theory}\label{sec:PrelimMeas}
Let $X$ be a compact topological space with Borel $\sigma$-algebra $\B$ and let
$T\colon X\to X$ be a Borel measurable map.
We will only consider measures
$\lambda$ for which every $B \in \B$ is measurable, and
such that $\lambda(U) > 0$ for every open set,
\ie $\lambda$ is fully supported Borel measure.
In particular, this means that if
$\eps > 0$, then $g(\eps) := \inf_{x \in X} \lambda(B(x;\eps)) > 0$.
Indeed, if $g(\eps)= 0$, then, due to compactness of $X$, we can find a convergent sequence $x_n \to x$ such that $\lambda(B(x_n;\eps)) \to 0$.
But then $B(x;\eps/2) \subset B(x_n;\eps)$ for $n$ sufficiently large,
so $\lambda(B(x; \eps/2)) \le \lim_n \lambda(B(x_n;\eps)) = 0$, in contradiction
to $\supp(\lambda) = X$.

\begin{defn}\label{def:exact}
A measure $\lambda$ on $\B$ is (with respect to $T$):
\begin{enumerate}
\item \emph{non-singular} provided that $\lambda(A)=0$ if $\lambda(T^{-1}(A))=0$,
\item \emph{conservative} if for every set $\lambda(A) > 0$ there is $n>0$ such that $\lambda(A\cap T^n(A))>0$,
\item \emph{ergodic}\footnote{
If Borel sets $A,B$ are such that $\lambda(A\setminus B)=0$ then we write $A\subset B \pmod \lambda$. Similarly, when $A\subset B \pmod \lambda$ and $B\subset A \pmod \lambda$ then we denote this fact by $A= B \pmod \lambda$.}
 if $A=T^{-1}(A) \pmod \lambda$
 implies that $\lambda(A)=0$ or $\lambda(A^c)=0$,
\item \emph{exact} if the {\em tail-field} $\bigvee_n T^{-n} \B$ is trivial,
or equivalently $\lambda(A)$ or $\lambda(A^c)  = 0$ whenever $T^{-n} \circ T^n(A) = A \pmod \lambda$ for all $n \in \N$.
\end{enumerate}
\end{defn}

\begin{rem}\label{rem:conservative}
If $\lambda$ is conservative, then $\lambda(A\cap T^n(A))>0$ for some $n>0$ and so there is $m>0$ such that
$$
\lambda(A\cap T^{n+m}(A))\ge \lambda(A\cap T^n(A)\cap T^m(A\cap T^n(A)))>0.
$$
In other words (by induction), for every $n>0$ there exists $k>n$ such that $\lambda(A\cap T^k(A))>0$, provided that $\lambda(A)>0$,
and hence it follows that $\lambda$-a.e.\ $x \in A$ returns to $A$ infinitely often.
\end{rem}

With any non-singular measure $\lambda$ w.r.t.\ $T$ we can associate the Perron-Frobenius operator $\TF \colon L^1(\lambda)\to L^1(\lambda)$,
uniquely defined by the formula
$$
\int_A \TF f d\lambda = \int_{T^{-1}(A)}f d\lambda \quad \text{ for all }f\in L^1(\lambda) \text{ and all }A\in \B.
$$

\begin{rem}\label{rem:exact}
An equivalent property to exactness (assuming that $\lambda$ is non-singular)
is that $\int |\TF^nf| d\lambda \to 0$ as $n \to \infty$
for any $f \in L^1(\lambda)$ with $\int f d\mu = 0$,
where $\TF$ is the Perron-Frobenius operator (see \cite{Lin} or \cite[Theorem 1.3.3.]{Aar}).
\end{rem}

\begin{lem}\label{L_Erg}
If $\lambda$ is a non-singular, $\sigma$-finite and exact Borel measure, then $d$-dimensional direct product measure $\lambda_d$
is ergodic for every integer $d\ge1$.
%Let $T$ be a non-singular map on the $\sigma$-finite measure
%space $(X,\mathcal{A},\mu)$ and $d\ge1$ an integer. If $T$ is exact, then the
%$d$-dimensional direct product system $(X^d, \mathcal{A}^d,\mu_d,T_d)$
%is ergodic.
\end{lem}

\begin{proof}
Parallel to Lemma 18 of \cite{BJ}.
\end{proof}

We repeat the following fact after Thaler \cite{thaler}.
\begin{lem}\label{fpconv0}
Let $\lambda$ be an exact non-singular Borel measure
and $\mu \ll \lambda$ an infinite $\sigma$-finite
$T$-invariant measure. Then
$$
\int_A \TF^n f d\lambda \to 0 \qquad \text{ as } n \to \infty,
$$
for all $A\in \B$, $\mu(A)<\infty$ and all $f\in L^1(\lambda)$.
\end{lem}

\begin{proof}
Fix $B\in \B$, $0<\mu(B)<\infty$ and denote
$$
g=\frac{\int_X f d\lambda}{\mu(B)}1_B \; h
$$
where $h=\frac{d\mu}{d\lambda}$ and observe that $\int_X (f-g)d\lambda=0$. Then
\begin{eqnarray*}
\int_A \TF^nf d\lambda &=& \int_A \TF^n(f-g) d\lambda+ \int_{T^{-n}(A)} g d\lambda\\
&= &\int_A \TF^n(f-g) d\lambda+\frac{\int_X f d\lambda}{\mu(B)}\mu(T^{-n}(A)\cap B)
\end{eqnarray*}
Since $\lambda$ is exact (see Remark~\ref{rem:exact}),
the first term tends to $0$.
By invariance of $\mu$ we obtain $\mu(T^{-n}(A)\cap B)\le \mu(T^{-n}(A))=\mu(A)$ and hence
$$
\limsup_{n\to \infty}\int_A \TF^nf d\lambda \le \frac{\mu(A)}{\mu(B)} \int_X f d\lambda .
$$
But $\mu$ is infinite and $\sigma$-finite, so we can start with arbitrarily large $\mu(B)$ which completes the proof.
\end{proof}

\begin{lem}\label{lem:exactnotfull}
Let $\lambda$ be an exact non-singular Borel measure
and $\mu \ll \lambda$ an infinite $\sigma$-finite
$T$-invariant measure.
Then $\lambda$ is not full.
\end{lem}

\begin{proof}
Since $\mu$ is $\sigma$-finite, there is $A\in \B$ with $0<\mu(A)<\infty$, in particular
$\lambda(A)>0$. While $\mu(X) = \infty$, there are countably many pairwise disjoint sets $F_i$ such that $\mu(F_i)<\infty$ and
$\bigcup_{i=1}^\infty F_i=X$.
Fix any $k>0$. Since $1=\lambda(X)=\sum_{i=1}^\infty \lambda(F_i)$ there is $r>0$
such that $\sum_{i=r+1}^\infty \lambda(F_i)< 2^{-k-1}$.
Hence, for every $n>0$ we have
$$
\lambda(T^{-n}(A))=\sum_{i=1}^\infty \lambda(T^{-n}(A)\cap F_i) \le 2^{-k-1}+\sum_{i=1}^r \lambda(T^{-n}(A)\cap F_i).
$$
But, for any $i$ we obtain by Lemma~\ref{fpconv0} that
$$
\lambda(T^{-n}(A)\cap F_i)=\int_{T^{-n}(A)} 1_{F_i}d\lambda = \int_A \TF^n(1_{F_i})d\lambda \to 0.
$$
Hence, there is $n_k>0$ such that $\lambda(T^{-{n_k}}(A)\cap F_i) < 2^{-k-1}/r$ for each $1\le i \le r$ and consequently
$\lambda(T^{-{n_k}}(A))<2^{-k}$.
In particular, if we set $B := \cup_{k \ge 1} T^{-n_k}(A)$ then $0<\lambda(B) <1$.
The proof is finished by the fact, that $\lambda(B^c) > 0$, but
$T^{n_k}(B^c) \cap A = \emptyset$ for all $k \ge 1$, which shows that
$\lambda$ is not full.
\end{proof}

Note that none of above properties required that $\lambda$ be
\emph{invariant} \ie $\lambda(A)=\lambda(T^{-1}(A))$ for any $A\in \B$.
%In the sequel we will denote invariant measures by $\mu$.
A definition that requires invariance is \emph{weak mixing}:
for every $A,B \in \B$, there is a sequence $n_k \to \infty$ such that
$\mu(T^{-n_k}(A) \cap B) \to \mu(A) \mu(B)$ as $k \to \infty$.
It is known that if $\mu$ is weak mixing, then for every $A,B \in \B$,
there is a set $\mathcal N\subset \N$ of full density such that
$\lim_{{\mathcal N} \owns n \to \infty} \mu(T^{-n}(A) \cap B) = \mu(A) \mu(B)$.

Let us recall two facts, which are Theorem 23 and Proposition 24 from
\cite{BJ}, respectively. We write $C^3_{nf}([0,1])$ for the collection of $C^3$
multimodal
interval maps $f\colon [0,1] \to [0,1]$ where all critical points $c$ are \emph{non-flat}, \ie
there is $\ell_c \in (1,\infty)$ such that $|T(x)-T(c)|/|x-c|^{\ell_c}$
is bounded and bounded away from zero for all $x$
sufficiently close to $c$.

For interval maps, we call
a closed invariant set $A \subset [0,1]$ an {\em attractor} (see \cite{milnor})
if its basin $\{ x \in [0,1] : \omega(x) \subset A\}$
has positive Lebesgue measure, and no proper subset of $A$ has this
property. Examples of attractors $A$ which are also Cantor sets
are the Feigenbaum attractor and the ``wild'' attractor
of a Fibonacci unimodal map of sufficiently high critical order,
see \cite{BKNS}.

\begin{thm}
Let $T \in C^3_{nf}([0,1])$ be a topologically mixing map having no
Cantor attractors. If $\lambda$ is conservative, then it is $\limsup$ full.
\end{thm}

\begin{thm}\label{thm:limfull}
Let $T \in C^3_{nf}([0,1])$ be topologically mixing and denote by $\lambda$ the Lebesgue measure on $[0,1]$. Then the following statements
are equivalent:
\begin{enumerate}
\item there exists an invariant probability measure $\mu \ll \lambda$,
\item $\liminf_{n\to\infty} ƒ\lambda(T^n(A)) > 0$ for every measurable set $A\in \B$, $\lambda(A)>0$,
\item $\lambda$ is full for $T$ .
\end{enumerate}
\end{thm}

This shows that Lebesgue measure is $\limsup$ full but not
full (for $T \in C^3_{nf}([0,1])$) precisely when it is conservative but admits no absolutely
continuous probability measure.
The first such examples (within the quadratic family) were constructed by Johnson \cite{Jon},
and more detailed constructions can be found in \cite{HK, Bthesis}.
Whether there exists a quadratic map for which $\lambda$ does not even admit a
$\sigma$-finite invariant probability measure is unknown.
However, there are cases where a $\sigma$-finite measure $\mu \ll \lambda$
exists such that $\mu(J)=\infty$ for all intervals $J \subset [0,1]$, see
\cite{Bthesis, ABJ, BJL}.

\begin{lem}
Let $\lambda$ be a fully supported non-singular probability measure
which is $\limsup$ full.
If $\lambda$ admits an equivalent $T$-invariant probability measure
$\mu$, then $\lambda$ is full.
\end{lem}

\begin{proof}
Let $g^+(\eps) = \sup\{ \mu(A) : \lambda(A) \le \eps\}$ and
%$g^-(\eps) = \inf\{ \mu(A) : \lambda(A) \ge \eps\}$
$g^-(\eps) = \sup\{ \lambda(A) : \mu(A) \le \eps\}$.
Since $\mu \ll \lambda$ and both $\mu,\lambda$ are probability measures, we easily obtain that $g^+(\eps) \to 0$ as $\eps \to 0$,
and since $\lambda \ll \mu$ also $g^-(\eps) \to 0$ as $\eps \to 0$.
Let $B$ be an arbitrary measurable set with $\lambda(B) > 0$,
and let $(n_k)$ be a sequence such that $\lambda(T^{n_k}(B)) > 1-1/k$.
Then $\mu(T^{n_k}(B)) \ge 1-g^+(1/k)$ and by invariance,
also $\mu(T^n(B)) \ge 1-g^+(1/k)$ for all $n \ge n_k$.
Using the definition of $g^-$, we find
$\lambda(T^n(B)) \ge 1-g^-(g^+(1/k))$ for all $n \ge n_k$.
Since $\lim_{\eps\to 0} g^+(\eps) = \lim_{\eps\to 0} g^-(\eps) = 0$, it follows that $\lambda$ is full.
\end{proof}

\section{Li-Yorke tuples and $d$-fold product measures}
\label{sec:dfold}
Another way of expressing the $LY_d$-tuples is
$$
LY_d = (\bigcap_{k=1}^\infty \bigcap_{r=1}^\infty \bigcup_{n>r} A_{k,n}) \cap
(\bigcup_{k=1}^\infty \bigcap_{r=1}^\infty \bigcup_{n>r} D_{k,n}),
$$
where
$$
\begin{cases}
A_{k,n} = \{ \x : \max_{i,j} \rho(T^n(x_i), T^n(x_j)) < 1/k\}, \\
D_{k,n} = \{ \x : \min_{i \neq j} \rho(T^n(x_i), T^n(x_j)) > 1/k \}.
\end{cases}
$$
Similarly, we can write
$$
LY_d^\delta = (\bigcap_{k=1}^\infty \bigcap_{r=1}^\infty \bigcup_{n>r} A_{k,n}) \cap
(\bigcap_{r=1}^\infty \bigcup_{n>r} D^\delta_{n}),
$$
where
$$
D^\delta_{n} = \{ \x : \min_{i \neq j} \rho(T^n(x_i), T^n(x_j)) > \delta \}.
$$

Since $A_{k,n}$ and $D_{k,n}$ and $D^\delta_{n}$ are open,
$LY_d$ and $LY_d^\delta$ are Borel sets
(in fact $G_\delta$-sets), hence their indicator function is measurable w.r.t.\ Borel measures on $X^d$, and we can use Fubini's theorem.

The following result is an extension of a result in \cite{BJ}
which states (for interval maps, but the argument is general) that if
$\lambda$ is $\limsup$ full, then the LY-pairs have full measure w.r.t.\ $\lambda_2$.
%\begin{thm}\label{t2}
%Let $\lambda$ be a non-singular, fully supported, Borel probability
%measure, and denote by $\lambda_d$ the $d$-fold product measure.
%\begin{enumerate}
%\item\label{t1:c1} If $\lambda$ is $\limsup$ full then $\lambda_2(LY^\delta_2)=1$ for any $\delta<\diam(X)/2$,
%\item\label{t2:c2} If $\lambda$ is full then $\lambda_d(LY_d^\delta)=1$ for
%every $d \ge 2$ and some $\delta>0$. If $X$ is additionally connected then
%$\lambda_d(LY_d^\delta)=1$ for every $\delta<\diam(X)/2(d-1)$.
%\end{enumerate}
%\end{thm}

\begin{proof}[Proof of Theorem~A]
We argue by induction. The initiation step is for $\delta$-LY-pairs with $\delta<\diam(X)/2$.
Given $x \in X$, let $LY_x = \set{ y \in X : (x,y) \in LY^\delta_2}$.
First we argue that $\lambda$-a.e.\ $y$ is $\delta$-separated in pair with $x$.
If this was not the case, the set of points which are not $\delta$-separated in pair with $x$ has positive measure, \ie
$$
\lambda \left(
\bigcup_{r=1}^\infty \bigcap_{n>r} \set{y : \rho(T^n(x),T^n(y))\le \delta}\right)>0.
$$
Therefore, we can find $r \in \N$ such
that $A := \cap_{n > r} \{ y : \rho(T^n(x), T^n(y)) \le \delta\}$ has positive measure.
Fix $0<\gamma< \diam(X)-2\delta$ and observe that
$\diam (T^n(A)) \le 2\delta < \diam(X)-\gamma$ for every $n>r$ and hence
there exists $x_n$
such that $T^n(A)\cap B(x_n;\gamma)=\emptyset$.
In particular $\lambda(T^n(A)) < 1-g(\gamma)<1$ for all $n > r$.
This contradicts the assumption that $\lambda$ is $\limsup$ full.

Similarly, we argue that $\lambda$-a.e.\ $y$ is proximal w.r.t.\ $x$.
Indeed, if not, then we can choose $k \in \N$ sufficiently
large such that the set $D := \cap_{n\ge 0} \{ y \in X : \rho(T^n(x), T^n(y)) > 1/k\}$
has positive measure.
But then $T^n(D) \cap B(T^n(x);1/k) = \emptyset$, and thus
$\lambda(T^n(D)) < 1-g(1/k)  < 1$ for all $n$, so again $\lambda$ is not $\limsup$ full.

The set of $\delta$-LY-pairs can be written as
$LY^\delta_2=\cup_x LY_x$, so by Fubini's theorem, it has full measure.
This completes the first step of induction and proves also
that if $\lambda$ is $\limsup$ full then $\lambda_2(LY^\delta_2)=1$ for any
$\delta<\diam(X)/2$.

We continue the induction, fixing $d\ge 2$ and assuming that $LY^\eps_d$
has full $d$-fold product measure for any $\eps<\diam(X)/2(d-1)$ when $X$ is connected, and for some $\eps>0$ otherwise
(hence every $\eps$ sufficiently small).
If $X$ is connected, then we fix any $0 < \delta<\diam(X)/2d$.
If $X$ is not connected, then we fix any distinct points $a_1,\ldots, a_{d+1}$
and put $\delta=\min\set{\min_{s\neq t} \rho(a_s,a_t)/6, \eps}$.

Take $\x \in LY^\delta_d$, let $(m_u)_{u \ge 1}$ and $(n_u)_{u \ge 1}$ be sequences
along which $\x$ is asymptotic, resp.\ separated, that is
$$
\begin{cases}
\liminf_u \max_{i,j} \rho(T^{m_u}(x_i), T^{m_u}(x_j)) = 0,\\
\limsup_u \min_{i\neq j} \rho(T^{n_u}(x_i), T^{n_u}(x_j)) > \delta.
\end{cases}
$$
and let
$$
LY_{\x} = \left\{ y \in X : \begin{array}{l}
\liminf_u \max_i \rho(T^{m_u}(x_i), T^{m_u}(y)) = 0, \\
\limsup_u \min_i \rho(T^{n_u}(x_i), T^{n_u}(y)) > \delta.
\end{array} \right\}
$$

%First, let us assume that $X$ is connected.
We show that $\lambda$-a.e.\ $y$ is $\delta$-separated with $\x$ along the
subsequence $(m_u)$.
%Indeed, if this is not the case, \ie
If the set of non $\delta$-separated points
has positive measure, then there is $r \in \N$ such
that $A := \bigcap_{u>r} \{ y : \min_i \rho(T^{m_u}(x_i), T^{m_u}(y)) \le \delta\}$ has positive measure.
But then $T^{m_u}(A) \subset \bigcup_{i=1}^d B(T^{m_u}(x_i); \delta)$.
If $X$ is connected, then take any $\xi>0$ such that $2d(\delta+2\xi)<\diam(X)$
and in the other case put $\xi=\delta$.

We claim that $\bigcup_{i=1}^d B(T^{m_u}(x_i);\delta+\xi)\neq X$.
First, we consider the case that $X$ is connected. If the claim does not hold then, since $X$ is connected, for every points $p,q\in X$
there are pairwise distinct numbers $i_1,\ldots, i_k$, where $k\le d$,
such that $p\in B(T^{m_u}(x_{i_1});\delta+\xi)$, $q\in B(T^{m_u}(x_{i_1});\delta+\xi)$ and $B(T^{m_u}(x_{i_j});\delta+\xi)\cap B(T^{m_u}(x_{i_{j+1}});\delta+\xi)\neq \emptyset$.
But then
\begin{eqnarray*}
\rho(p,q)&\le& \delta+\xi + 2(k-1)(\delta+\xi)+\delta+\xi \le 2d\delta + 2d\xi\\
&\le& 2d(\delta+2\xi)-2d\xi < \diam(X) - 2\xi.
\end{eqnarray*}
which is a contradiction, since $p,q$ were arbitrary. Indeed, the claim holds.
Similarly, if $X$ is not connected, then by the definition of $\delta$,
each ball $B(T^{m_u}(x_{i_{j+1}});\delta+\xi)$ can contain at most one point $a_s$,
and hence also in this case $\bigcup_{i=1}^d B(T^{m_u}(x_i);\delta+\xi)\neq X$. Indeed, the claim holds.

By the above claim, for every $u$ there is a point $q_u$ such that
$$
B(q_u;\xi)\cap \bigcup_{i=1}^d B(T^{m_u}(x_i);\delta)=\emptyset
$$
which in particular implies that $\lambda(T^{m_u}(A)) < 1-g(\xi)$ for all $u > r$.
This contradicts that $\lambda$ is full. This proves that $\lambda$-a.e.\ $y$ is $\delta$-separated with $\x$ along the
subsequence $(m_u)$.

Similarly, we argue that $\lambda$-a.e.\ $y$ is proximal w.r.t.\ $\x$ along the subsequence $(n_u)$.
Indeed, otherwise we can choose $k \in \N$ sufficiently
large such that the set $D := \bigcap_{u=0}^\infty \{ y \in X : \max_i \rho(T^{n_u}(x_i), T^{n_u}(y)) \ge 1/k\}$
has positive measure. Passing to a subsequence of $n_u$ if necessary
we may assume that for every $u$ there is $j$ such that if $y\in D$ then $\rho(T^{n_u}(x_j), T^{n_u}(y))\ge 1/k$.
But then $T^{n_u}(D) \cap (B(T^{n_u}(x_j);1/k)) = \emptyset$ and so
$\lambda(T^{n_u}(D)) < 1-g(1/k)  < 1$ for all $u$,
therefore again $\lambda$ is not full.

The set of $\delta$-LY $(d+1)$-tuples contains the set
$LY^\delta_{d+1}\supset \bigcup_{\x \in LY^\delta_d} LY_{\x}$, so again by Fubini's theorem, it has full measure.
This completes the proof.
\end{proof}

\begin{lem} If the $d$-fold product $(X^d, T_d)$
has a conservative non-atomic product measure $\lambda_d$,
then the set of $d$-tuples that are not asymptotic has
full $\lambda_d$-measure.
\end{lem}

\begin{proof}
Assume on the contrary, that the set of asymptotic $d$-tuples has positive measure, \ie
\begin{equation}
\label{eq:As}
\lambda_d(\bigcap_{k>0}\bigcup_{r>0}\bigcap_{n>r}A_{k,n})=\alpha>0.
\end{equation}
where $A_{k,n}=\set{\x : \max_{i,j} \rho(T^n(x_i), T^n(x_j)) < 1/k}$.
Take $\eps>0$ sufficiently small, such that if we denote by $\Delta_\eps$ the $\eps$-neighbourhood
of the diagonal $\Delta = \{ \x : x_i = x_j \text{ for all } 1 \le i,j \le d \}$
then $\lambda_d(\Delta_\eps)<\alpha/2$.
Fix any $k \in \N$ such that $1/k<\eps$. By \eqref{eq:As}
there is $m>0$ such that
$$
\lambda_d(\bigcup_{r=1}^m \bigcap_{n>r}A_{k,n})>\alpha/2.
$$
Denote $A=\bigcup_{r=1}^m \bigcap_{n>r}A_{k,n}=\bigcap_{n>m}A_{k,n}$.
Note that
$T_d^n(A) \subset \Delta_\eps$ for all $n \ge m$ and $\lambda_d(A')>0$
for $A' = A \setminus \Delta_\eps$. By conservativity, $\lambda(T_d^n(A') \cap A') > 0$ for some $n>m$
which contradicts the definition of $A$.
Therefore the set of asymptotic $d$-tuples is of measure zero and the lemma follows.
\end{proof}

The following is a simple extension of a well-known fact for LY-pairs
(see \eg \cite{BHR}):

\begin{lem}\label{lem:weak_mixing}
If $X$ has at least two points and $\mu$ is a weakly mixing, fully supported, $T$-invariant Borel
measure, then there is $\delta_d>0$ such that the $\delta_2$-LY $d$-tuples have full $\mu_d$-measure for any integer $d >1$.
\end{lem}

\begin{proof}
We start with proximality.
Assume by contradiction that there is a set
$A \subset X^d$ with $\mu_d(A) > 0$, and
$m \in \N$, $\eps > 0$ such that
$T_d^n(A) \cap \Delta_\eps = \emptyset$ for every $n \ge m$.
Since $\mu$ is fully supported, $\mu_d(\Delta_\eps) > 0$.
But $\mu_d$ is weak mixing, hence the product measure $\mu_d$ is weak mixing too, and so
there is a subsequence $(n_k)$ such that
$\mu_d(T_d^{-n_k}(\Delta_\eps) \cap A) \to \mu_d(\Delta_\eps) \mu_d(A)>0$.
This implies that there is $n>m$ such that $T_d^n(A)\cap \Delta_\eps\neq \emptyset$, contradicting the definition
of $A$.
Therefore $\mu_d$-a.e.\ $\x$ is proximal along a subsequence.

Since $X$ has at least two points and $\mu$ is weakly mixing, $X$ is infinite.
In particular, there exists an open set $U\subset X^d$ such that
$\overline{U}\cap \Delta_{i,j}=\emptyset$ for every $i\neq j$, where $\Delta_{i,j} :=
\{ \x \in X^d : x_i = x_j\}$. Take any $0<\delta < \inf_{\x\in U}\min_{i\neq j} \rho(x_i,x_j)$.
We use the same argument, to show that a.e.\ $\x$ visits $U$ infinitely often under action of $T_d$.
Combining the two, we obtain that  $\mu_d$-a.e.\ $d$-tuple is $\delta$-LY.
\end{proof}

\begin{rem}\label{rem:delta_2}
It is clear from the proof that the statement of Lemma~\ref{lem:weak_mixing}
holds for any $\delta_2<\diam (X)$.
\end{rem}

\section{Li-Yorke tuples in the continuous setting}\label{sec:continuous}

%\begin{defn}
Let $T$ be a continuous map on compact metric space $(X,\rho)$.
We say that $T$ is \emph{topologically mixing} if for every pair of open sets $U,V \subset X$,
$T^n(U) \cap V \neq \emptyset$ for all $n$ sufficiently large.
We say that $T$ is {\em topologically weak mixing} if $T\times T$ is transitive on $X\times X$.
%\end{defn}

The following fact is standard and its utility to Li-Yorke chaos dates back at least to works of Iwanik \cite{I}.

\begin{lem}\label{lem:inv-d-dist}
Let $(X,\rho)$ be a compact metric space with at least two points and let $T\colon X \to X$ be topologically weakly mixing.
For every $d>1$ there is $\delta_d>0$ such that the set of $\delta_d$-LY $d$-tuples is residual.
%\begin{enumerate}[(i)]
%\item
%For every $d>1$ there is $\delta_d>0$ such that the set $D_d$ consisting of tuples $(x_1,\ldots,x_d)$ satisfying for any $s_1,\ldots,s_d\ge 0$ condition
%\begin{equation}
%\limsup_n \min_{i\neq j} \rho(T^{n+s_i}(x_i), T^{n+s_j} (x_j)) > \delta_n
%\label{d-inv-dist}
%\end{equation}
%is residual in $X^d$.
%\item For every $d>1$ the set $P_d$ of tuples $(x_1,\ldots,x_d)$ satisfying for any $s_1,\ldots,s_d\ge 0$ condition
%\begin{equation}
%\liminf_n \max_{i, j} \rho(T^{n+s_i}(x_i), T^{n+s_j} (x_j)) = 0
%\label{d-inv-prox}
%\end{equation}
%is residual in $X^d$.
%\end{enumerate}
\end{lem}

\begin{proof}
Since $X$ is not a singleton, topological mixing implies
that $X$ is an infinite set without isolated points. Fix a sequence of pairwise distinct points $\set{a_i}_{i=1}^\infty$
and let $\delta_d=\inf_{1\le i< j\le d} \rho(a_i,a_j)/3$.
Now, any $d$-tuple with orbit dense in $X^d$ is $\delta_d$-LY and the set of such tuples is residual by transitivity of $T_d$.
%Fix any $d>1$ and pick any $d$ distinct points $z_1,\ldots,z_d\in X$.
%Put $\delta_d=\max_{i\neq j} \rho(z_i,z_j)/4$.
%Fix $k>0$ and integers $s_1,\ldots, s_d\ge 0$. Take any open sets $U_1,\ldots,U_d\in X$.
%Since $X$ has no isolated points, there are pairwise disjoint open sets $V_i\subset U_i$.
%By weak mixing, there is $n>k$ such that
%$T^{-n-s_i}(B(z_i,\delta_d))\cap V_i\neq \emptyset$.
%Therefore, the set of $d$-tuples $\x=(x_1,\ldots, x_d)$:
%$$
%D_{k,s_1,\ldots,s_d}=\set{\x : \max_{i\neq j} d(T^{n+s_i}(x_i),T^{n+s_j}(x_j))>2\delta_d \text{ for some }n>k}
%$$
%is open and dense. But then the $G_\delta$-set
%$$
%\tilde{D}_d=\bigcap_{k>0}\bigcap_{s_1,\ldots,s_d\ge 0} D_{k,s_1,\ldots, s_d}
%$$
%is residual, while it is not hard to see that any tuple in $\tilde{D}_d$ satisfies the condition \eqref{d-inv-dist}.
%This proves that $D_d$ is a residual set.
%
%The proof that the set $P_d$ is residual follows the same lines.
\end{proof}

Let us recall an important fact which can be derived from works of Kuratowski and Mycielski (see \eg \cite{Myc64}).
We recall that $M\subset X$ is called a {\em Mycielski set}
if it is a countable union of Cantor sets.

\begin{thm}[Kuratowski-Mycielski]\label{Mycielski}
 Let $X$ be a perfect complete metric space, and assume that
$R_k$ is a residual subset of $X^{n_k}$, where $n_k\ge 2$ for each $k\in \N$. Then there exists a Mycielski set $M$ dense in $X$ such that for each $k\in \N$ if points $x_1,\cdots, x_{n_k}\in M$ are pairwise distinct then $(x_1,\cdots, x_{n_k})\in R_k$.
\end{thm}

The following fact is known, but since the proof is simple, we sketch it for completeness.

\begin{lem}\label{lem:dense_tup_LY}
Let $(X,\rho)$ be a compact metric space with at least two points and let $T\colon X \to X$ be topologically weakly mixing.
There is a sequence $\set{\delta_d}_{d=2}^\infty \subset (0,1)$ and a Mycielski set $M\subset X$ such that for any $d\ge 2$, any $d$ pairwise distinct points $x_1,\ldots, x_d\in M$ and any integers $s_1,\ldots,s_d\ge 0$ the $d$-tuple
$(T^{s_1}(x_1),\ldots, T^{s_d}(x_d))$ is $\delta_d$-LY.
Additionally, $\omega(x,T)=X$ for every $x\in M$.
\end{lem}
\begin{proof}
Combine the technique from the proof of Lemma~\ref{lem:inv-d-dist} with Theorem~\ref{Mycielski} and the fact that the set of points with
dense orbit is residual in $X$
to obtain a desired Mycielski set.
\end{proof}

Let $\Sigma_2^+=\set{0,1}^\N$ be endowed with the standard prefix metric, and let
$\sigma$ be the left shift on $\Sigma_2^+$.
The following fact was proved by Moothathu in \cite{tksm}
(which extends the existence of horseshoes argument from
Misiurewicz \& Szlenk \cite{MSz} to the $C^0$ setting ensuring that
the map $\pi$ is really injective everywhere).

\begin{lem}\label{lem:conj}
If a map $T$ acting on the unit interval has positive topological entropy, then
there exist $n>0$, a $T^n$-invariant closed set $\Lambda$
and a homeomorphism $\pi\colon \Lambda \to \Sigma_2^+$ such that $\pi \circ (T^n|_{\Lambda})=\sigma \circ \pi$.
\end{lem}

\begin{thm}\label{constr:m1m2}
Let $T\colon [0,1] \to [0,1]$ be topologically mixing. There exist dense Mycielski sets $M_1,M_2,M_3$ and numbers $\delta_n>0$ such that:
\begin{enumerate}
\item each $n$-tuple consisting of $n\ge 2$ distinct points in $M_i$ is $\delta_n$-LY, where $i=1,2$,
\item every point in $M_1$ has dense orbit in $[0,1]$,
\item there exists a minimal Cantor set $A$ such that $\omega(x,T)=A$ for every $x\in M_2$ and $M_2\cap A$ contains a Cantor set,
\item $M_3\times M_3$ contains no Li-Yorke pairs.
\end{enumerate}
\end{thm}
\begin{proof}
The set $M_1$ is obtained by a direct application of Theorem~\ref{lem:dense_tup_LY}. Constructing the set $M_2$ requires a little more work.

It is known that every mixing map on the unit interval has positive topological entropy, so by Lemma~\ref{lem:conj} we can find $n>0$
and a $T^n$-invariant set $\Lambda$ where the dynamics is conjugated with the full one-sided shift $\Sigma_2^+$.
Clearly, we may assume that $\Lambda \subset (0,1)$.
Take any minimal weakly mixing subset of $\Sigma_2^+$ (\eg one-sided version of the Chac\'on flow \cite{BlaKwi})
and let us denote it by $A_n$.
Passing through conjugating homeomorphism, we may assume that $A_n\subset \Lambda$. Clearly it is a Cantor set as a perfect subset of the
Cantor set $\Sigma_2^+$ (up to conjugating homeomorphism) and also it is not hard to see that
$A=\bigcup_{j=0}^{n-1} T^j(A_n)$ is a minimal subsystem of $T$. Let us apply Theorem~\ref{lem:dense_tup_LY} to the dynamical system
$T^n|_{A_n}$, and let numbers $\delta_d>0$ and a Cantor set $C\subset A_n \subset A$ be such that
for any $d>1$, any $d$ pairwise distinct points $x_1,\ldots, x_d\in C$ and any integers $s_1,\ldots,s_d\in n\N\cup \set{0}$ the tuple $(T^{s_1}(x_1),\ldots, T^{s_d}(x_d))$ is $\delta_d$-LY for $T^n$ (clearly, it is also $\delta_d$-LY for $T$).

Since $C$ is homeomorphic to $C^2$ we can find pairwise disjoint Cantor sets $\set{C_i}_{i=0}^\infty$ such that $\cup_i C_i\subset C$.
There is also $\eps>0$ such that $C\subset (\eps,1-\eps)$. Let $\set{U_i}_{i=1}^\infty$ be a sequence composed of all open subintervals of $[0,1]$
with rational endpoints. Since $T$ is mixing, for every $i$ there is $k_i$ such that $T^{n k_i}(U_i)\cap [0,\eps]\neq \emptyset$ and $T^{n k_i}(U_i)\cap [1-\eps,1]\neq \emptyset$. It is known that if an image of a compact set via a continuous map is uncountable then there is a
Cantor set on which this map is one-to-one (see \eg \cite[Remark 4.3.6]{Siv}).
Hence, for every $i$ there is a Cantor set $D_i\subset U_i$ such that $T^{nk_i}|_{D_i}$ is one-to-one and $T^{n k_i}(D_i)\subset C_i$.
Denote $M_2 = C_0 \cup \bigcup_{i=1}^\infty D_i$. Since for every $x\in M_2$ there is $j\ge 0$ such that $T^j(x)\in A$
and $A$ is a minimal set, we immediately obtain that $\omega(x,T)=A$ for every $x\in M_2$.

Let us take any tuple $(x_1,\ldots, x_d)$ of pairwise distinct points in $M_2$. There are numbers $i_1,\ldots, i_d$
such that $x_j\in D_{i_j}$. Denote $z_j=T^{n k_{i_j}}(x_j)\in C_j$ and observe that points $z_j$ are also pairwise distinct.
Let $m=\max_{j} k_{i_j}$ and denote $s_j=m-k_{i_j}$. Now, it is enough to note that
\begin{eqnarray*}
\rho(T^{nm}(x_p), T^{nm} (x_q))&=&\rho(T^{n(s_p+k_{i_p})}(x_p), T^{n(s_q+k_{i_q})}(x_q))\\
&=&\rho(T^{ns_p}(z_p), T^{ns_q}(z_q))
\end{eqnarray*}
The construction of $M_2$ completed by the definition of the set $C$.

To construct $M_3$, let us first recall that every Sturmian minimal system does not contain Li-Yorke pairs \cite[Example~3.15]{BGKM} (it is a so-called almost distal system)
and $\Sigma_2^+$ contains an uncountable family of pairwise disjoint Sturmian systems \cite[Proposition~4.44]{Kurka}.
Repeating conjugacy argument from the previous step we may view Sturmian minimal subshift $M$ (which is a Cantor set) as a subset of $[0,1]$.
Note that since $M$ is an infinite minimal system for $T^n$ without Li-Yorke pairs, set $\hat{M}=\bigcup_{j=0}^{n-1}T^j(M)$ is also minimal, and cannot contain Li-Yorke pairs (sets $T^j(M), T^i(M)$ for $i\neq j$ are either disjoint of equal).

Proceed the same way as in the construction of $M_2$ to obtain a dense Mycielski set $M_3$ such that
every point $x\in M_3$ is eventually transformed into a point $z_x\in M$. Now, take any $x,y\in M_3$
and $k$ sufficiently large, so that $T^k(x),T^k(y)\in \hat{M}$. If $T^k(x)\neq T^k(y)$ then $(x,y)$ is not Li-Yorke pair by the definition of
$\hat{M}$ and if $T^k(x)=T^k(y)$ then $(x,y)$ is not Li-Yorke pair neither.
\end{proof}

As a direct consequence of theorem by Oxtoby \& Ulam (see \cite[Thm.\ 9]{Ulam})
we obtain that if $B\subset [0,1]$ is a dense Mycielski set then there exists a
homeomorphism $\phi \colon [0,1] \to [0,1]$ such that
$\phi(B)$ has full Lebesgue measure.

In fact, if $A,B\subset [0,1]$ are dense in $[0,1]$
and  either of them is the union of pairwise disjoint Cantor sets,
then there is an increasing
homeomorphism $\varphi\colon [0,1]\to [0,1]$ such that $\varphi (A)\subseteq B$.

Note that the dense Mycielski sets $M_1, M_2, M_3$ in Theorem~\ref{constr:m1m2} are pairwise disjoint.
In particular, if one of them has full Lebesgue measure, the other have measure zero.
As a direct consequence of Theorem~\ref{constr:m1m2} we obtain Theorem~C.

%\begin{thm}\label{thm:coexist_anota}
%There exist pairwise topologically conjugate maps $P, S, T\colon [0,1] \to [0,1]$
%and sets $K, L, M$ such that:
%\begin{enumerate}
%\item all sets $K,L,M$ have full Lebesgue measure;
%\item there is a sequence of positive numbers $\delta_d$ such that every $d$-tuple of distinct points in $L$ (resp. in $M$)
%is $\delta_d$-LY for $S$ (resp.\ $T$);
%\item $\omega(x,S) = [0,1]$ for every $x \in L$;
%\item\label{anota:c4} there exists a Cantor minimal set $A$ of positive measure such that $\omega(x,T)=A$ for every $x\in M$.
%\item none of the pairs $(x,y)\in K\times K$ is LY for $P$; in particular the set of LY pairs  has zero Lebesgue measure.
%\end{enumerate}
%\end{thm}

The above theorem shows that in interval dynamics Cantor attractor and no Cantor attractor cases
always coexist, but their "physical" visibility depends on the special structure of the map.
In fact, using the sets $M_1, M_2, M_3$ from Theorem~\ref{constr:m1m2}, we can distribute
Lebesgue measure in any proportion between Li-Yorke pairs made of points with dense orbits, Li-Yorke pairs with a Cantor
attractor, or without Li-Yorke pairs.

\begin{rem}
If $f\in C^3_{nf}(I)$ then by Theorem~E in \cite{vSV}, any minimal set must have zero Lebesgue measure.
Hence, the situation described in Theorem~C(\ref{anota:c4}) can never occur for these maps.
\end{rem}

%{\color{red}P: Can conjugacy from Theorem~\ref{thm:coexist_anota} be obtained in $C^n$ setting (n=1,2,3)?
%HB: I'm pretty sure it cannot. One problem I see is with $C^2$ smoothness is that invariant Cantor sets have Lebesgue measure
%zero. In the $C^1$ topology, probably more can be done, at the price of
%a lot of technicalities.}

\section{The Manneville-Pomeau map}\label{sec:MP}

\subsection{Inducing}
Suppose that we are in the case when every non-singular Borel measure $\mu \ll \lambda$ for $T$ is $\sigma$-finite and infinite.
Then Theorem~\ref{thm:limfull} indicates that $\lambda$ is not full,
so the results of Section~\ref{sec:dfold} are not applicable in this case.
Another approach to address the question of Li-Yorke $d$-tuples
is via inducing.
The idea is to choose an appropriate subset $Y \subset X$ and consider
the (first return) induced map $(Y, F)$,
such that $Y$ can be decomposed in
countably many subsets $\mathcal{Z} := \{ Y_j \}_{j \in \N}$ with
$\lambda(Y \setminus \cup_j Y_j) = 0$ and such that
$F(Y_j) = T^{\tau_j}(Y_j) = Y$ where the first return
time $\tau(x) = \min\{ n \ge 1 : T^n(x) \in Y\}$ is constant
$\tau_j$ of $Y_j$.
Assume that the {\em distortion} of the Jacobian
is bounded uniformly in the iterate $n$ \ie
\begin{equation}\label{eq:distortion}
\sup_{n\in\N} \ \sup_{Z \in \bigvee_{i=0}^{n-1} F^{-i} \mathcal{Z} }
\ \sup_{x,y \in Z}\  \frac{ J_{F^n}(x) }{ J_{F^n}(y) } < \infty.
\end{equation}
We take the Jacobian w.r.t.\ Lebesgue measure $\lambda$:
For any $j$, the push-forward measure $\lambda(F(A))$ for $A\subset Y_j$ is
well-defined and absolutely continuous with respect to $\lambda$.
Hence,  $J_{F} = \frac{d\mu \circ F}{d\lambda}$ is well-defined
for $\lambda$-a.e.\ $x \in Y_j$. Similarly, we can define $J_{F^n}$ on any
set $Z \in \bigvee_{i=0}^{n-1} F^{-i} \mathcal{Z}$.
By the definition $Y = \cup_j Y_j \pmod 0$ and sets $Y_j$ are in practice intervals, hence we may view each Jacobian $J_{F^n}$ as a function defined on $Y$.

In the case of bounded distortion $(Y,F)$ preserves an absolutely continuous invariant
probability measure. Let us call this measure $\nu$;
its density $\frac{d\nu}{d\lambda}$ is bounded and bounded away from $0$.
It projects to a $T$-invariant measure
\begin{equation}\label{eq:pullback}
\mu(A) = \sum_n \sum_{i=0}^{n-1} \nu(T^{-i}(A) \cap \{ y \in Y :  \tau(y) = n\}),
\end{equation}
which can be normalised if the normalising constant
$\int_Y \tau \ d\nu < \infty$, but if not, then $\mu$
is $\sigma$-finite, see \eg \cite[Chapter 6]{BG}.
The measure $\nu$ is ergodic and exact, and
and these properties carry over to $(X,T,\mu)$
(for exactness to carry over, we need the additional condition
that $gcd(\tau) = 1$). In fact, $\mu|_Y$ and $\nu|_Y$ differ by a fixed
constant, regardless whether $\int_Y \tau \ d\nu < \infty$ or not,
because $F:Y \to Y$ is the first return map, see \eqref{eq:pullback}.

If the tail $\nu(\{ y : \tau(y) > n\})$ is sufficiently heavy
(and regular), then the probability of two independently chosen initial
points to return to $Y$ at the same time infinitely
 often under iteration of $T$ can be zero.

Let us denote
$$
u_n = \lambda(\{ y \in Y : T^n(y) \in Y  \})
$$
and observe that
$$
\lambda_d(\{ y \in Y^d : T_d^n(y) \in Y^d  \})=u_n^d.
$$
Since $\frac{d\mu}{d\lambda}$ is bounded and bounded away from zero on $Y$,
we can freely interchange $\mu$ and $\lambda$ in these formulas. Therefore, if
 \begin{equation}\label{eq:square-summable}
 \sum_n u_n^d < \infty,
 \end{equation}
then the Borel-Cantelli Lemma gives
\begin{eqnarray*}
0&=&\mu_d(\bigcap_{m=1}^\infty \bigcup_{n=m}^\infty \{ y \in Y^d : T_d^n(y) \in Y^d  \})\\
&=&\mu_d(\{ x \in Y^d : T_d^n(x) \in Y^d \text{ infinitely often}\}).
\end{eqnarray*}

The following recurrence lemma seems to be standard (see \cite[Proposition 1.2.2]{Aar}).
However we could not find exact reference in the literature and hence decided to provide a proof for completeness.

\begin{lem}\label{L_Rec}
Let $\mu$ be a $\sigma$-finite, non-singular and $T$-invariant Borel measure.
If there exists $Y\in \B$ such that
$0<\mu(Y)<\infty$ and
\begin{equation}\label{Eq_DualSumsDiverge}
\sum_{n\ge0}\TF^n1_Y=\infty   \qquad \mu\text{-a.e. on } Y
\end{equation}
(where $\TF$ is the Perron-Frobenius operator), then $Y$ is a recurrent set in the sense that
\begin{equation}\label{Eq_RecurrentSet}
Y\subseteq\bigcup_{n\ge1}T^{-n}(Y) \pmod \mu.
\end{equation}
Moreover, if $\mu$ is ergodic and $T$-invariant, then $\mu$ is conservative.
\end{lem}

\begin{proof}
First, we need to show that the set $A:=Y\setminus\bigcup_{n\ge 1} T^{-n}(Y)$ of points which never
return to $Y$ has measure zero. We must have $A\cap T^{-j}(A)=\emptyset$
for every $j \ge 1$ because $A\subseteq Y$.
This immediately implies that $A$ is a wandering set for $T$, that is, the preimages
$T^{-n}(A)$, $n\ge0$, are pairwise disjoint.  Directly from the definition of Perron-Frobenius operator $\TF$
we have
$\int_{A}\TF^{n}1_{Y}\,d\mu=\mu(Y\cap T^{-n}(A))$, hence
$$
\int_{A}\sum_{n \ge 0} \TF^n1_{Y}\,d\mu=\sum_{n \ge 0}\mu(Y\cap
T^{-n}(A))\le \mu(Y)<\infty,
$$
By assumption \eqref{Eq_DualSumsDiverge} we obtain that $\mu(A)=0$ which ends the proof of
\eqref{Eq_RecurrentSet}.

As a consequence of \eqref{Eq_RecurrentSet} we see that the set $E:=\bigcup_{n\ge
0}T^{-n}(Y)$ satisfies $T^{-1}(E)=E \pmod \mu$. Since
$Y\subset E$, we have $\mu(E)>0$. If $T$ is ergodic, then we can conclude that
$X=\bigcup_{n\ge0}T^{-n}(Y) \pmod \mu$. This, since $T$ is measure preserving transformation, implies that assumptions of Maharam's
Recurrence Theorem are satisfied (see \eg \cite[Thm 1.1.7]{Aar}) which thus ensures that $T$ is conservative.
\end{proof}

\begin{prop}\label{prop:cons}
Let us assume that $T$ preserves a non-singular exact Borel probability measure $\mu$ equivalent to $\lambda$.
If there exists a set $Y \subset X$ recurrent in the sense of \eqref{Eq_RecurrentSet} such that $\mu(Y)>0$
and the first return map $(Y,F)$ has only onto branches
with bounded distortion (in the sense of \eqref{eq:distortion})
then the following conditions are equivalent for every integer $d\ge 1$:
\begin{enumerate}
\item\label{prop:cons:1} $\sum_n u_n^d = \infty$, where $u_n = \lambda(\{ y \in Y : T^n(y) \in Y\})$,
\item\label{prop:cons:2} $d$-fold product measure $\lambda_d$ is
ergodic and conservative.
\end{enumerate}
\end{prop}

\begin{proof}
To prove $\eqref{prop:cons:2}\Longrightarrow\eqref{prop:cons:1}$ let us first observe, that if $\lambda_d$ is conservative then
\begin{eqnarray*}
0<\lambda_d(Y^d)&=&\lambda_d(\{ x \in Y^d : T_d^n(x) \in Y^d \text{ infinitely often}\})\\
&=&\lambda_d(\bigcap_{m=1}^\infty \bigcup_{n=m}^\infty \{ y \in Y^d : T_d^n(y) \in Y^d  \})\\
\end{eqnarray*}
and hence by the Borel-Cantelli Lemma (see \eqref{eq:square-summable}) we must have $\sum_n u_n^d = \infty$.

To prove $\eqref{prop:cons:1}\Longrightarrow\eqref{prop:cons:2}$,
denote by $\TF$ and $\FF$ the Perron-Frobenius operator for $T$ and $F$, respectively, both w.r.t.\ $\lambda$.
Then for every $n>0$ and a.e.\ $x\in Y$ we have
\begin{equation}
\FF^n f(x) =\sum_{y\in F^{-n}(x)} \frac{f(y)}{J_{F^n}(y)},\label{density:form}
\end{equation}
where $J_{F^n}$ is the Jacobian w.r.t.\ $\lambda$.
The bounded distortion of $J_{F^n}$ applied to \eqref{density:form} allows us to verify that there is
a constant $\kappa>0$ such that for every measurable $A \subset X$ and  $n>0$ we have
\begin{equation*}\label{Eq_BasicEstimate}
\inf_{x\in A}\sum_{k=0}^n \FF^k1_{Y}(x)\ge\kappa\cdot\sup_{x\in A}\sum_{k=0}^n \FF^k 1_Y(x).
\end{equation*}
A similar estimate for the $d$-fold product system with Perron-Frobenius
operator $\FF_d$,
\begin{equation*}
\inf_{\x\in A}\sum_{k=0}^n\FF_{d}^{k}1_{Y^d}(\x)\ge\kappa^{d}\cdot
\sup_{\x\in A}\sum_{k=0}^n \FF_{d}^{k}1_{Y^d}(\x).
\end{equation*}
Hence, if there is a set $A\subset Y$ of positive measure such that $\sum_{k=0}^n \FF_{d}^{k}1_{Y^d}(\x)$
is uniformly bounded for all $n \ge 0$
and $\x \in A$, then $\sum_{k=0}^n \FF_{d}^{k}1_{Y^d}(\x)$ is uniformly bounded for all $n \ge 0$
and $\lambda$-a.e.\ $\y\in Y$.

Observe that for any Borel set $A\subset Y^d$ we have
$$
\sum_{k=0}^n\int_{A}  \TF_{d}^{k}1_{Y^{d}}\ d\lambda_d\le
\sum_{k=0}^n\int_{A} \FF_{d}^{k}1_{Y^{d}}\ d\lambda_d
\le \sum_{k=0}^\infty \int_{A}  \TF_{d}^{k}1_{Y^{d}}\ d\lambda_d.
$$
If there exists a set $A\subset Y^d$ with $\lambda_d(A)>0$ and $M>0$ such that
$\sum_{k\ge 0}\TF_{d}^{k}1_{Y^{d}}(\x)<M$ for every $\x\in A$ then there is $c>0$ such that
$$
M\ge \sum_{k\geq 0}\int_{A} \FF_{d}^{k}1_{Y^{d}}\ d\lambda_d
\ge c \sum_{k\geq 0}\int_{Y^d} \FF_{d}^{k}1_{Y^{d}}\ d\lambda_d
\ge c \sum_{k\geq 0}\int_{Y^d} \TF_{d}^{k}1_{Y^{d}}\ d\lambda_d.
$$
Consequently, divergence of
$\sum_{k\ge0}u_{k}^{d}=\sum_{k\ge0}\int_{Y^d} \TF_{d}^k 1_{Y^d}\,d\lambda_d$
implies that
$$
\sum_{k\ge 0}\TF_{d}^{k}1_{Y^{d}}=\infty \qquad \lambda\text{-a.e. on } Y^{d}.
$$
By Lemma~\ref{L_Erg}, $d$-dimensional direct product measure $\lambda_d$, is ergodic with
respect to $T_d$ for every $d\ge 1$.
Recall that $\mu_d$ is a $T_d$-invariant, non-singular and $\sigma$-finite measure. But $\lambda_d$ is ergodic and $\mu_d$ is equivalent to $\lambda_d$, hence
$\mu_d$ is ergodic. Applying Lemma~\ref{L_Rec} to $\mu_d$ and $T_d$,
we obtain that $\mu_d$ is conservative, and
using once again equivalence of $\mu_d$ and $\lambda_d$ we conclude that $\lambda_d$ is conservative, which completes the proof.
\end{proof}

\subsection{Manneville-Pomeau maps}
The classical example from interval dynamics where the $u_n$ can be
computed is the Manneville-Pomeau family, where Lebesgue measure $\lambda$ will be our reference measure for $\mu$.
These are interval maps with a neutral fixed point and the inducing is with respect to a set $Y$ bounded away from this fixed point.
For us, it is convenient to use the
family $T_\alpha:[0,1] \to [0,1]$ defined by
$$
T_\alpha(x) = \begin{cases}
x(1+2^\alpha x^\alpha) & \text{ if } x \in [0,\frac12), \\
2x-1 & \text{ if } x \in [\frac12,1] =: Y. \\
\end{cases}
$$
It has an indifferent fixed point at $0$, and
the first return map $F:Y \to Y$ is uniformly expanding
with bounded distortion (uniformly in all iterates, in the sense
of \eqref{eq:distortion}, see \eg \cite{luzzatto}),
$F:Y \to Y$ preserves an ergodic measure which pull back to an
ergodic and conservative $T_\alpha$-invariant measure
$\mu$, see Theorem~1 and Corollary~2 in \cite{ThaEstim}.

\begin{rem}\label{rem:Thaler}
More precise estimates on the invariant density $h$ were given by
Thaler \cite[Corollary~1]{ThaEstim}), who showed that
$h(x) x^{\alpha}$ is bounded and bounded away from $0$.
In addition (see \cite[Lemma~2.3]{LSV}) for $\alpha\in (0,1)$ the density $h = \frac{d\mu}{d\lambda}$
is Lipschitz outside a neighbourhood of $0$ and $\mu$ is mixing (see \eg \cite[Theorem~7]{luzzatto}).
\end{rem}

If $y_0 = \frac12$ and $y_{k+1}$ is the unique point in $T_\alpha^{-1}(y_k) \cap [0,y_n]$, then $y_n \searrow 0$ such that
\begin{equation*}\label{eq:yk}
y_n \sim  n^{-\beta} \qquad \text{ for } \beta = \frac{1}{\alpha},
\end{equation*}
where $a_n \sim b_n$ stands for $\lim_n a_n/b_n \in (0, \infty)$, see \cite{dB}.
If $y'_{n+1} \in [\frac12, 1]$ is the other preimage of $y_n$,
then $[\frac12, y'_{n+1}) = \{ y \in Y : \tau(y) \ge n+2\}$.
By Remark~\ref{rem:Thaler},
$\mu(\{ y \in Y : \tau(y) \ge n \}) \sim |y'_{n-2}-\frac12|
= \frac12y_{n-1} \sim n^{-\beta}$ and therefore
\begin{eqnarray}\label{eq:infinity}
\int \tau d\mu &=& \sum_n n \mu(\{ y \in Y : \tau(y) = n \}) \nonumber \\[2mm]
&=& \sum_n \mu(\{ y \in Y : \tau(y) \ge n \})
\begin{cases}
< \infty & \text{ if } \alpha \in (0,1);\\
= \infty  & \text{ if } \alpha \ge 1.
\end{cases}
\end{eqnarray}

\begin{thm}\label{thm:asymptotics}
Assume that $\alpha > 1$ and write $\beta = 1/\alpha \in (0,1)$.
Then $u_n \sim n^{\beta-1}$.
In particular,
$\sum_n u_n^d = \infty$ if and only if $\alpha \le \frac{d}{d-1}$.
\end{thm}

\begin{proof}
From the above considerations, we have
for the Manneville-Pomeau map that
$$
\mu(x \in Y : \tau(x) = n) \sim
\lambda([y_n, y_{n-1}])
\sim (n-1)^{-\beta} - n^{-\beta} \sim n^{-(1+\beta)}.
$$
Then the estimate on $u_n$ is a special case of the results of
Gou\"ezel \cite{Gouezel},
partly correcting results from Doney \cite{D} (see \cite[Section 1.3]{Gouezel}).
More precisely, \cite[Proposition 1.7]{Gouezel} applied to $u = v = 1_Y$
gives $u_n = \int_Y u \cdot v \circ T^n d\mu \sim n^{\beta-1}$.
Now it follows immediately that
$\sum_n u_n^d = \infty$ if and only if $\alpha \le \frac{d}{d-1}$.
\end{proof}

%\begin{thm}\label{t58}
%For the Manneville-Pomeau map $T_\alpha$, the following conditions hold:
%\begin{enumerate}
%\item $\lambda_2$-a.e.\ pair is $\frac13$-Li-Yorke,
%\item and for $d \ge 3$:
%\begin{enumerate}[(i)]
%\item if $\alpha \le \frac{d-1}{d-2}$ and $\delta < 1/2(d-2)$, then $\lambda_d$-a.e.\ $d$-tuple is $\delta$-Li-Yorke;
%\item if $\alpha > \frac{d-1}{d-2}$ and $\eps > 0$, then
%$\lambda_d$-a.e.\ $d$-tuple is not $\eps$-separated;
%in particular $\lambda_d$-a.e.\ $d$-tuple is not Li-Yorke.
%\end{enumerate}
%\end{enumerate}
%\end{thm}

\begin{proof}[Proof of Theorem~B]
Recall that by Remark~\ref{rem:Thaler} for $\alpha < 1$
the map $T_\alpha$ preserves a mixing absolutely continuous
probability measure,
and hence the result follows by Lemma~\ref{lem:weak_mixing} (see also Remark~\ref{rem:delta_2}).

If $\alpha \ge 1$, then there is an infinite $\sigma$-finite
measure $\mu \sim \lambda$ (by Remark~\ref{rem:Thaler}, the density of $\mu$ is bounded and bounded away from $0$).
More precisely, using Remark~\ref{rem:Thaler} again,
$\mu([y_k, 1]) < \infty$ for fixed $k$,
but $\mu([y_l,y_k]) \to \infty$ as $l \to \infty$.
This means that given a neighbourhood $U$ of the fixed point $0$,
Lebesgue-a.e.\ point $x$ spends
almost every iterate in $U$  (\ie $N(x,U) := \set{n \ge 0: T^n(x)\in U}$.
has density $1$ for $\nu$-a.e. $x$).
Indeed, for $U = [0,y_k)$, we have
\begin{eqnarray*}
\lim_{n\to\infty} \frac1n \#\{ 0 \le i < n : T^i(x) \in [y_k, 1]\}
&\le& \lim_{n\to\infty} \frac{ \#\{ 0 \le i < n : T^i(x) \in [y_k, 1]\} }
{ \#\{ 0 \le i < n : T^i(x) \in [y_l, 1]\} } \\
&=& \frac{\mu([y_k,1])}{\mu([y_l,1])} \to 0 \text{ as } l \to \infty,
\end{eqnarray*}
by the Ratio Ergodic Theorem (see \eg \cite[Theorem 2.2.5.]{Aar}).
It follows that $\lambda_d$-a.e.\ $d$-tuple is proximal along a
subsequence.

Next, we claim that every pair is LY.
This follows easily from the fact that $T_\alpha$ is expanding
away from $0$. More precisely,
for every two $x,y \in [0,1]$ which are
not eventually mapped to one another,
$\rho(T_\alpha^i(x), T_\alpha^i(y)) \le \frac13$ implies that
$\rho(T_\alpha^{i+1}(x), T_\alpha^{i+1}(y)) > \rho(T_\alpha^i(x), T_\alpha^i(y))$.
Therefore $(x,y)$ is $\frac13$-separated.

Now let $d \ge 3$,  $1 \le \alpha < \frac{d-1}{d-2}$ and choose
$\delta < 1/2(d-2)$.
By Theorem~\ref{thm:asymptotics} applied to $d-1$ coordinates,
$\sum_n u_n^{d-1} = \infty$ and hence,
$([0,1]^{d-1}, \lambda_{d-1}, T_{d-1})$ is conservative by
Proposition~\ref{prop:cons}.
Therefore, taking $A \subset Y^{d-1}$ with $\mu_{d-1}(A) > 0$,
for $\mu_{d-1}$-a.e.\ $\underline x$, there is a sequence $(t_n)$ such that
$(T_{d-1})^{t_n}(x_i) \in A$ for all $n \in \N$.
The set
$$
A = [\frac12, \frac12+\eta] \times [\frac12+\frac1{2(d-2)},
\frac12+\frac1{2(d-2)}+\eta] \times \dots \times [1-\eta, 1]
$$
has positive $\lambda_{d-1}$-measure, and clearly
points $\underline y \in A$ are $\frac{1}{2(d-2)}-2\eta$-separated
$d-1$-tuples.
Choosing $\eta$ so small that $\frac{1}{2(d-2)}-2\eta > \delta$ then shows that
the $d-1$-tuple $\x$ is $\delta$-separated along a subsequence
$(t_n)$.

We claim that the remaining coordinate $x_d$ is close to $0$, or
more precisely: for $\lambda$-a.e.\ $x_d$
we have $T_\alpha^{t_n}(x_d) < \frac12-\delta$ infinitely often.
Indeed, take $U_N := \{ x_d \in [0,1] :
T_\alpha^{t_n}(x_d) \ge \frac12-\delta \text{ for all } n \ge N\}$
and let $x_d^* \in U_N$ be arbitrary.
Let $H_n \owns x_d^*$ be the maximal interval such that $T_\alpha^{t_n}(H_n)
= [\frac12(\frac12-\delta),1]$ and $H'_n \subset H_n$ is such
that  $T_\alpha^{t_n}(H'_n) =  [\frac12(\frac12-\delta), \frac12-\delta)$.
Since $t_n$ is increasing, we clearly have $\lim_{n\to\infty}|H_n|=0$.
The maps $T_\alpha^{t_n}:H_n \to [\frac12(\frac12-\delta),1]$
have uniformly bounded distortion (this is proved in virtually the same way
as the distortion bound for the branches of $F^k$ is proven,
cf.\ \cite{luzzatto}).
Therefore there is $K > 0$ such that $|H'_n| > |H_n|/K$ for all
$n \in \N$. Since $U_N \cap H'_n = \emptyset$ for $n \ge N$,
it follows that $x_d^*$ cannot be a Lebesgue density point of $U_N$.
This means that $\lambda(U_N) = 0$ and hence
$\lambda(\cup_N U_N) = 0$ as well, proving the claim.
Using Fubini's Theorem, we conclude that $\lambda_d$-a.e.\ $d$-tuple is indeed
$\delta$-separated along a subsequence.

Finally, if $\alpha > \frac{d-1}{d-2}$, typical $d-1$-tuples
$\x = (x_1, \dots x_{d-1})$
visit $Y$ simultaneously only finitely often.
Now let $\eps = y_k$ and assume by contradiction that
$\x$ visits $[\eps, 1]^{d-1}$ infinitely often
for a set of $\x$ of positive $\lambda_{d-1}$-measure.
Hence, there are integers $a_1, \dots a_{d-1} \in \{ 0, 1, \dots, k\}$
such that $T^{n+a_1}(x_1), \dots, T^{n+a_{d-1}}(x_{d-1}) \in Y$
for infinitely many $n$, still for a set $U$ of positive $\lambda_{d-1}$-measure.
Take $a = \max_i a_i$ and for each $\x \in U$
take a $d-1$-tuple $\y$ with coordinates
$y_i \in T^{a_i-a}(x_i) \cap Y$. Since $T$ is non-singular (and the Cartesian
product $\prod_{i=1}^{d-1} T^{a-a_i}$ is non-singular too), we have
$\lambda_{d-1}(\{ \y : \x \in U\}) > 0$.
Now for every $\y$,
there is an infinite sequence $(n_k)_{k \in \N}$ (depending on $\underline y$ but
not on the index $i = 1,\dots, d-1$), such that
$T^{n_k}(y_i) = T^{n_k+a_i-a}(x_i) \in Y$
for each $i$ and $k$.
This contradicts the first statement of this paragraph.

Coming back to a typical $d$-tuple $(x_1, \dots, x_d)$,
the above argument shows that no matter
how we select a $d-1$-tuple $\x$ from it, for all sufficiently
large $n$, at least one coordinate $T^n(x_i) \in [0,\eps)$.
Therefore at least two coordinates of the $d$-tuple belong to $[0,\eps)$.
Since $\eps > 0$ can be taken arbitrary small, the proof is complete.
\end{proof}

\begin{rem}
We can replace
the right branch of the Manneville-Pomeau map by $2(1-x)$ if a continuous map is
preferred. The dynamical properties that we are concerned with
remain the same. The same is true, if we view Manneville-Pomeau map as a continuous map on the unit circle.
\end{rem}

\begin{rem}
One can increase the number of neutral fixed point, \eg define
the map
$$
T_{\alpha, \beta}(x) = \begin{cases}
x+2^\alpha x^{1+\alpha} & \text{ if } x \in [0,\frac12), \\
x-2^\beta(1-x)^{1+\beta} & \text{ if } x \in [\frac12,1], \\
\end{cases}
$$
see Figure~\ref{fig:linear_chain2},
and consider the Li-Yorke behavior of tuples for this map.
If $\alpha > \beta > 1$, then neutral fixed point $0$ dominates, and we expect
the same behaviour as in
%Theorem~\ref{t58}.
Theorem~B.
For the case $\alpha = \beta$, we expect that
$\lambda_3$-a.e.\ $3$-tuple is Li-Yorke
(where for typical triples $(x_1, x_2, x_3)$, there are infinitely
many $n$ with $T^n(x_1) \approx 0$, $T^n(x_2) \approx 1$
 and $T^n(x_3) \in [\frac13, \frac23]$, as well as
are infinitely many $m$ with $T^m(x_1) \approx T^m(x_2) \approx T^m(x_3) \approx 0$.
Conjecturally, for $d \ge 4$, typical $d$-tuples are Li-Yorke
if and only if $\alpha \le \frac{d-2}{d-3}$.
\end{rem}

\begin{figure}[htb]
\begin{center}
\includegraphics[width=0.3\textwidth]{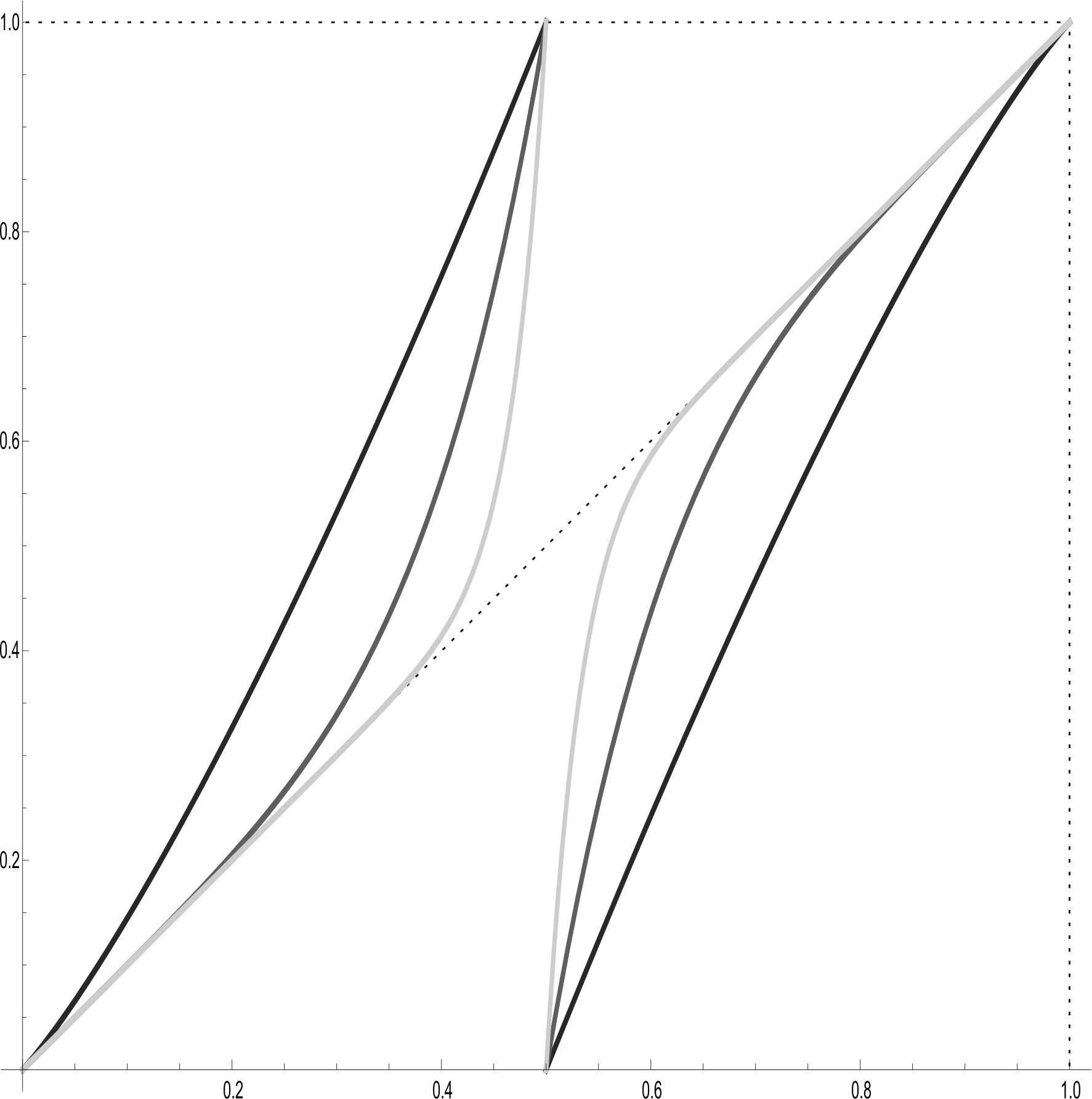}
\caption{Graph of $T_{\alpha,\beta}$ for $\alpha = \beta = \frac12, 4, 15$.}
\label{fig:linear_chain2}
\end{center}
\end{figure}

\section*{Acknowledgements}
The first author wants to thank Dalia Terhesiu and Roland Zweim\"uller for sharing their expertise on the infinite measure preserving
examples in this paper.
Also the support of OeAD % \"Osterreichischer Austauschdienst
(Project Number: PL 02/2013) and
MNiSW (Project Number: AT 2/2013-15), as well as the hospitality of the Max Planck Institute for Mathematics in Bonn,
are gratefully acknowledged.

\end{document}